\documentclass{amsart}
\usepackage{graphicx}
\usepackage{cite}
\usepackage{url}

\newtheorem{thm}{Theorem}[section]

\theoremstyle{definition}

\theoremstyle{remark}
\newtheorem{rem}[thm]{Remark}
\numberwithin{equation}{section}
\begin{document}

\title[On a normality criterion of S. Mandelbrojt ]{On a normality criterion of S. Mandelbrojt }%
\author{P.V.Dovbush}%
\address{Institute of Mathematics and Computer Science  of  Academy of Sciences  of Moldova, 5 Academy  Street,  Kishinev, MD-2028, Republic
of Moldova}%
\email{peter.dovbush@gmail.com}%
\thanks{I would like to thank Professor Robert B. Burckel for all his help during my work on this paper.}%
\subjclass[2010]{32A19}%
\keywords{Normal families, Holomorphic functions of several complex variables.}%
%\date{}%
%\dedicatory{}%
%\commby{}%

% ----------------------------------------------------------------
\begin{abstract}
Extension of classical Mandelbrojt's criterion for normality of a family of holomorphic zero-free functions of several complex variables is given. We show that a family of holomorphic functions of several complex variables whose corresponding Levi form are uniformly bounded away from zero is normal.
\end{abstract}
\maketitle
\section{Introduction and Main Result}

In 1929, Mandelbrojt \cite[p.189]{Mandelbrojt1929}  asserted his criterion for normality of a family of holomorphic zero-free functions of one complex variables. That criterion can be considered as a variation of Montel's Theorem.

The main purpose of this paper is to establish an extension of Mandelbrojt's theorem to several complex variables.

Let $\mathcal F$ be a  family of zero-free holomorphic functions in a domain $\Omega$ and $D$ be a subdomain in $\Omega$ such that
$\overline{D}\subset\Omega.$ So that  the quantities
$$m(f,D)=\left\{ \begin{array}{llll} \sup\limits_{z,w\in D} \frac{\ln |f(z)|}{\ln|f(w)|}, & \mbox{ if } & |f(w)|\neq 1 \textrm{ for all } w \textrm{ in } D;\\
\infty, & \mbox{ if } & |f(w)|= 1  \textrm{ for some } w \textrm{ in } D, \end{array} \right. $$
$$m'(f,D)=\sup_{z,w\in D}\frac{|f(z)|}{|f(w)|},$$
$$L(f,D)=\min [ m(f,D),m'(f,D)],$$
are well defined for each function $f \in \mathcal F.$

Our main result is as
follows:

\begin{thm} \label{Mandelbrojt} Let $\mathcal F$ be a family of zero-free holomorphic functions in a domain $\Omega.$ Then $\mathcal F$ is normal in $\Omega$
if and only if for each point $z_0 \in \Omega$ there exists a ball $\overline{B(z_0,r_0)}\subset \Omega$ such that the  the set of quantities
$L(f,B(z_0,r_0)),$ $f\in \mathcal F,$ is bounded.
\end{thm}

\section{Preliminaries}

A family $\mathcal F$ of holomorphic functions on a domain $\Omega \subseteq C^n,$ $n\geq 1,$ is normal in $\Omega$ iff every sequence of function $\{f_n\}\subseteq \mathcal F$ contains either a subsequence which converges to a limit function $f \not \equiv \infty$ uniformly on each compact subset of $\Omega,$ or a subsequence which converges uniformly to $\infty$ on each compact subset.

A family $\mathcal F$ is said to be normal at a point $z_0 \in \Omega$ if it is normal in some neighborhood of $z_0.$

The connection between these definitions is given by
 \begin{thm}  \label{t2.1} A family of holomorphic functions $\mathcal F$ is normal in a domain $\Omega \subset \mathbf C^n$ iff
$\mathcal F$ is normal in at each point of $\Omega.$
\end{thm}
The proof of this theorem is the same as the proof of Theorem 3.2.1 in \cite{Schiff1993}.

Normality of a sequence of zero-free holomorphic functions imposes a tight restriction on limit functions, as the following theorem of Hurwitz's shows.
 \begin{thm} (Hurwitz's theorem). On a connected open set,
the normal limit of nowhere-zero holomorphic functions of several complex variables is either nowhere zero or identically
equal to zero.
\end{thm}

On normality of a family of holomorphic functions of several complex variables, many results are
obtained by G.Julia \cite{Julia1926}. From them, we denote the following interesting
theorem:

\begin{thm} \label{t2.2} (Fundamental Normality Test) Let $\mathcal F$ be a family of holomorphic functions in a domain in $\mathbf C^n$ if
every function $\mathcal F$ does not take two fixed different values then $\mathcal F$ is normal.
\end{thm}
At first sight it is surprising  to see that a condition on the images of a family could
guarantee normality.
However this result is not surprising if it is considered in
light of the uniformization theorem  and the Poincar\'{e} metric on the
domain $\mathbf C \setminus \{a,b\}$.

Let $\mathcal F$ be a family of holomorphic functions in a domain $\Omega.$ We said that $\mathcal F$ to be locally bounded in $\Omega$ if for each point $a\in \Omega$ we can find a ball $B(a,r):=\{z\in \mathbf C^n :|z-a|<r\}$ belonging to $\Omega$  and a positive number $M>0$ such that for each functions $f \in \mathcal F,$ the inequality  $|f(z)|<M $ holds in $B(a,r).$

\begin{thm} \label{t2.3} (Montel's Theorem). Let $\mathcal F $  be a family of
holomorphic functions on $\Omega.$  If $\mathcal F $  is locally uniformly bounded in $\Omega,$ then $\mathcal F $ is normal in $\Omega.$
\end{thm}

In case of several complex variables Theorem \ref{t2.2} and \ref{t2.3} are an easy consequence of Montel's Theorem \cite[p. 35]{Schiff1993} and Fundamental Normality Test \cite[p. 54]{Schiff1993} from one complex variables, Theorem \ref{t2.1} and the following
theorem of Alexander \cite[Theorem  6.2]{Alexander1974}.
\begin{thm} \label{t2.4} A  family  $\mathcal F$  of  analytic  functions  on  open unit ball $B:=B(0,1) \subset \mathbf C^n$ is  normal  if  and  only  if
the  restriction  of $\mathcal F$   to  every complex  line  through  the  origin  is normal.    A  local  version  at  the
origin  of  the  radial  theorem  also  holds.
\end{thm}

It is well known that a family of holomorphic functions can be normal without being equicontinuous or locally uniformly bounded (see for example \cite[Example 2.2.3]{Schiff1993}).

In contradicting of Theorem \ref{t2.3}, where local boundedness was the key ingredient,  the normality of holomorphic functions is characterized by a condition in which the Levi form of corresponding functions is locally uniformly bounded on compact subsets (see Timoney \cite{Timoney1980} or Dovbush
\cite[Lemma 1]{Dovbush1981}).

\begin{thm} \label{cm} (Marty's Theorem) A family of holomorphic functions $\mathcal F$ is normal in a domain $\Omega \subset \mathbf C^n$ if and only if
for each compact set $K\subset \Omega$ there is a constant $C=C(K)$ such that for all $f \in \mathcal F$  at each point $z \in K$ the Levi form $L_z(\log(1+|f|^2),v)$ satisfy
\begin{equation}\label{eq1}
L_z(\log(1+|f|^2),v):= \sum^n_{\mu,\nu}\frac{\partial^2\ln(1+|f(z)|^2)}{\partial z_\mu \partial\overline{z}_\nu}v_\mu \overline{v}_v\leq C|v|^2 \textrm{ for all } v \in \mathbf C^n .
\end{equation}
\end{thm}

Unfortunately, in practice Marty's criterion almost useless,
as verification of the condition (\ref{eq1}) in cases when normality is not already evident is
generally extremely difficult.

Marty's Theorem provides a complete and satisfying answer to the question of
when a family of functions is normal. But if given a family of holomorphic functions of several complex variables
$\{f\}$ and a constant $c>0$  such that for every compact subsets $K\subset\Omega$ we have
$$L_z(\log(1+|f|^2),v) \geq c|v|^2  \textrm{ for all } z \in K, v \in \mathbf C^n ,$$ we see that the Marty's criterion is insufficient
to establish normality.

Denote by $\chi(\cdot,\cdot)$ the chordal distance on $\overline{\mathbf C}$ and by $\delta(\cdot, \cdot)$ the spherical distance on $\overline{\mathbf C}.$ For the definitions and properties of the chordal and spherical distance see, for example, Schiff \cite[pp. 2-3]{Schiff1993}.

If $f$ is a holomorphic function of one complex variables by $f^\sharp(\lambda)$  we denote the spherical derivative
$$f^\sharp(\lambda):=\lim_{h\to 0}\frac{\chi(f(\lambda+h),f(z))}{h}=\frac{|f'(\lambda)|}{1+|f(\lambda)|^2}.$$

\begin{thm} \label{suf} Let some $c > 0$ be given and set
\begin{multline*}\mathcal F:=\{f \textrm{ holomorphic in }\Omega : L_z(\log(1+|f|^2),v) \geq c|v|^2  \textrm{ for all } z \in \Omega, v \in \mathbf C^n \}.\end{multline*}
Then $\mathcal F$ is normal in $\Omega.$
\end{thm}
\begin{proof}Since normality is a local property, we can restrict all our considerations
concerning normal families to the unit ball.
We draw an arbitrary complex line $\l_\zeta(\lambda)= \{z\in \mathbf C^n : z= \lambda \zeta\}$ through the point $0\in \mathbf C^n,$ where $\zeta \in \partial B$ (i.e., $|\zeta|=1$) is arbitrary but fix, and $\lambda \in \mathbf C;$ its intersection with $B$ in the $\lambda$-plane obviously correspond to the unit disk $\Delta=\{|\lambda|=1\}.$ If $f \in \mathcal F$ then the function $\ln(1+|f|^2)$ is $C^2$ in $B,$ then by chain rule, for its restriction  to this line, the spherical derivative
$$|f^\sharp(\lambda \zeta)|^2=L_{\lambda \zeta}(\log(1+|f|^2),\zeta) \geq c|\zeta|^2=c.$$
It follows by \cite[Theorem 1]{J.Grahl2012} that the  restriction  of $\mathcal F$ to  each  complex  line  $\l_\zeta$  is  normal. By Theorem \ref{t2.4} $\mathcal F$ is normal in $B.$
\end{proof}

\section{Proof of main result}

Observe that the chordal distance between circles $|w'|=1$ and
$|w''|=2$ is given by $\chi(w',w'')=1/\sqrt{10},$ as does also the chordal distance between the two
circles $|w'|=1$ and $|w''|=1/2.$ Then for any two points  $w',w''$ that satisfy either the two relations $|w'|\leq 1$ and $|w''|\geq 2,$
or the two relations $|w'|\geq 1$ and $|w''|\leq 1/2,$ it must be true that $\chi(w',w'')\geq 1/\sqrt{10}.$

Indeed, the function
        $$ g(t): = \frac{1-t}{\sqrt{2 + 2t^2}}, \qquad 0 < t < 1,$$
   visibly satisfies
    \begin{equation}\label{rb}
     g \text{ is decreasing on the interval } [0,1]
    \end{equation}
   as it is the product of the two positive decreasing functions $1-t$ and
   ${1}/\sqrt{2 + 2t^2}.$
   Now, for complex numbers $w'$ and $w'',$ with $|w'| < 1 < 2 < |w''|,$
     \begin{multline*}\chi(w',w'') \geq \frac{|w''| - |w'|}{\sqrt{(1 + |w'|^2)(1 + |w''|^2)}}> \\
     \frac{|w''| - 1}{\sqrt{2(1 + |w''|^2)}}
             = g(1/|w''|) > \\ g(1/2) [\text{ due to (\ref{rb}) }] =1/\sqrt{10}.
     \end{multline*}
From the symmetry of
   the chordal metric in its two variables it is easy to see that $$\chi(w',w'')>1/\sqrt{10}$$ holds if the point $w'$ belongs
 to the complement of the closed unit disc and $w''$ belongs to the open disc of radius $1/2.$

\begin{proof}[Proof of Theorem \ref{Mandelbrojt}] $\Rightarrow$  Suppose  $z_0$  is an arbitrary point in $\Omega.$  Let $r$ be such that $\overline{B}(z_0,r)\subset\Omega.$ Suppose $\mathcal F$ is normal in $\Omega.$  If $f \in \mathcal F$ its restriction to an arbitrary complex line $\{z \in \mathbf C^n : z=z_0+\lambda v, \lambda \in \mathbf C, |\lambda|<r, |v|=1\},$ i.e., the function
$h_v(\lambda)=f(z_0+\lambda v),$ is holomorphic in disc $\Delta_r=\{\lambda \in \mathbf C: |\lambda|<r\}$ and from (\ref{eq1}) follows that its spherical derivative
$$(h_v^\sharp(\lambda))^2=L_{z_0+\lambda v}(\ln(1+|f|^2),v)<C|v|^2 \textrm{ for all } \lambda \in \overline{U}_r,$$
so that
$$\delta^2(h_v(0),h_v(\lambda))<C|\lambda v|^2  \textrm{ for all } \lambda \in \overline{U}_r.$$
It follows
$$\delta(f(z_0),f(z))<C|z-z_0| \textrm{ for all } z \in {B}(z_0,r)$$
and we conclude that $\mathcal F$ is spherically equicontinuous at point $z_0.$ Since $\chi(w',w'')\leq \delta(w',w'')<\frac{\pi}{2}\chi(w',w'')$ the family $\mathcal F$ is also chordal equicontinuous at $z_0.$
 It follows that about each point $z_0 \in \Omega,$ there is a ball
${B(z_0,r)}\subseteq \Omega$ in which
\begin{equation}\label{po}
\chi(f(z),f(z_0))<1/\sqrt{10}
\end{equation}
for all $f \in \mathcal F$ and for all $z$ in $B(z_0,r).$

Let $z_0$ be arbitrary, but fixed point in $D.$ If, say, $|f(z_0)|\leq 1$ then $|f(z)|< 2$ for $z\in B(z_0,r),$ $f \in \mathcal F,$ (otherwise,
$\chi(f(z),f(z_0))>\chi(1,2),$ a contradiction with the inequality (\ref{po})) whereas $|f(z_0)|> 1,$
then $|f(z)|> 1/2$ (otherwise, $\chi(f(z),f(z_0))>\chi(1,1/2)=\chi(1, 2),$ a contradiction with the inequality (\ref{po})).
Therefore, if  $z\in B(z_0,r),$ either
$$ |f(z)|< 2 \textrm{ or } \frac{1}{|f(z)|}<2, \ \ \ \ f \in \mathcal F.$$
It follows that $\mathcal F$ can be expressed as the union of two families
$$\mathcal J:=\{f \in \mathcal F:|f(z_0)|\leq 1, |f(z)|<2, z\in B(z_0,r)\},$$
$$\mathcal H:=\{f \in \mathcal F: \frac{1}{|f(z_0)|}<1, \frac{1}{|f(z)|}<2, z\in B(z_0,r)\}.$$

Suppose first that $\{f_j\}\in \mathcal F$ and $f_j \to 0$ as $j \to \infty.$  Then $|f_j(z)|<1/2$ for all $z\in \overline{B(z_0,r)}$ and $j$  sufficiently
large.
Set $$\mathcal J_1:=\mathcal J\setminus \{f \in \mathcal J, |f(z)|<1/2, z\in B(z_0,r)\}.$$
There exists a constant $\varepsilon>0$ such that $\min \{|f(z)|, z\in \overline{B(z_0,r)}\}>\varepsilon $ for all $f \in \mathcal J_1.$ We  argue  by
contradiction  and  suppose  otherwise.  Then there exist a sequence of functions $\{f_k\},$ $f_k \in \mathcal J_1,$  and a sequence of points $z_k \in
\overline{B(z_0,r)}$ such that
$f_k(z_k)\to 0$ as $k \to \infty$ and $z_k\to z_1\in \overline{B(z_0,r)}$ as $k \to \infty.$ Since $\mathcal J_1$ is a normal family of zero-free functions
$\{f_k(z)\}$ has a subsequence $\{f_l(z)\}$
which converges uniformly on $\overline{B(z_0,r)}$  to a  function $f$ holomorphic in $\Omega.$ Since
$$|f_l(z_l)-f(z_1)|\leq |f_l(z_l)-f(z_l)|+|f(z_l)-f(z_1)|$$ it follows that $f(z_1)=0.$ By multi-dimensional version of Hurwitz's theorem  $f\equiv 0.$
Hence $|f_l(z)|<1/2$ for all $z \in B(z_0,r)$ and all $l$  sufficiently large. This
is a contradiction with the assumption that $f_l(z)\in \mathcal J_1.$

 If  $f(z)\in \mathcal J_1$  and since $f$ is zero-free holomorphic function in $\Omega$ we have
 $$\frac{\varepsilon}{2}\leq\frac{|f(z)|}{|f(w)|}\leq \frac{2}{\varepsilon}$$
  for all points $z$ and $w$ in $\overline{B(z_0,r)}$ and all $f \in \mathcal J_1.$ (The first inequality also holds because $z$ and $w$ play symmetric
  roles.)

If  $f \in \mathcal J,$ and if $ |f(z)|<1/2$ for all $ z\in B(z_0,r)$ then $1/f$ is holomorphic in $\Omega$ because $f$ never vanishes; moreover
$|(1/f)(z)|>2$ on    $B(z_0,r).$ Hence $\ln|(1/f)(z)|$  is  a positive pluriharmonic in $B(z_0,r).$   Pluriharmonic functions form a subclass of the class
of harmonic functions in  $B(z_0,r)$ (obviously proper for $n > 1$). So by
Harnack's inequality  there exists some  constant $C=C(\overline{B(z_0,r_0)},B(z_0,r)),$ $C\in (1,\infty),$ $r_0<r,$ such that
$$\frac{1}{C}\leq\frac{\ln(1/|f(z)|)}{\ln(1/|f(w)|)}\leq C \textrm{ for all } z \textrm{ and } w \textrm{ in } \overline{B(z_0,r_0)}.$$
 Hence
 $$ \frac{1}{C}\leq\frac{\ln|f(z)|}{\ln|f(w)|}\leq C \textrm{ for all } z \textrm{ and } w \textrm{ in } \overline{B(z_0,r_0)}.$$

Since the same proof works for functions in $\mathcal H$ we have for all $f \in \mathcal H$ either
$$\frac{\varepsilon}{2}\leq\frac{|f(z)|}{|f(w)|} \leq \frac{2}{\varepsilon} \textrm{ for all } z \textrm{ and } w \textrm{ in }
\overline{B(z_0,r_0)}$$
or
$$  \frac{1}{C}\leq\frac{\ln|f(z)|}{\ln|f(w)|}\leq C \textrm{ for all } z \textrm{ and } w \textrm{ in } \overline{B(z_0,r_0)}.$$

 Hence the set of quantities $L(f,B(z_0,r_0)),$ $f \in \mathcal F,$
is bounded  and we obtain the desired result.

$\Leftarrow$ Fix  an arbitrary point $z_0$ in $\Omega$ and define the families
$\mathcal J$ and $\mathcal H$ by
$$\mathcal J : =\{f \in \mathcal F : |f(z_0)|\leq 1\},$$
$$\mathcal H : =\{f \in \mathcal F : |f(z_0)|>1\}.$$
  We first show that $\mathcal J$ is normal at $z_0.$ Let a sequence $\{f_j\}$ in $\mathcal J$ be given. First choose $r_0 > 0$
    small enough that the last line in the statement of the theorem holds. The following two cases exhaust all the possibilities for sequence $\{f_j\}:$
\begin{itemize}
\item[ (e)] there exists a subsequence $\{f_{j_k}\}$ such that for any $k\in\mathbb N$ the function $\log|f_{j_k}|$ does not vanish in~$B(z_0,r_0)$;
\item[ (f)] for each $j\in\mathbb N$ there exists $z_j\in B(z_0,r_0)$ such that $\log|f_j(z_j)|=0$.
\end{itemize}

In case  (e) we have that  $|f_{j_k}|<1$ in $B(z_0,r_0)$ for all elements of the sequence. Such a subsequence is normal in $B(z_0,r_0)$ by Montel's theorem and hence we are done in case  (e).

In case  (f) we have $m(f_j,B(z_0,r_0))=+\infty$ for all $j\in\mathbb N$. Therefore, according to the hypothesis,  $m'(f_j,B(z_0,r_0))<C$ for all $j\in\mathbb N$ and some   finite constant $C.$  It follows that  $|f_j|<C|f_j(z_0)|\leq C$ in $B(z_0,r_0)$ for all $j\in\mathbb N$, which means that $\{f_j\}$ is a normal family in $B(z_0,r_0)$ and hence finishes the proof in case (f).

 If $f \in \mathcal H$  then $1/f$ is holomorphic
on $\Omega$ because $f$ never vanishes. Also $1/f$ never vanishes and $1/|f(z_0)| < 1.$ Hence reasoning similar to that in the above proof shows that
$\mathcal {\widetilde{H}} : =\{1/f: f\in \mathcal H\} $ is also normal in  $B(z_0,r_0).$ So if $\{f_j\}$ is a sequence in $\mathcal H$
there is a subsequence $\{f_{j_k}\}$ and a holomorphic function $h$ on $B(z_0,r_0)$ such that $\{1/f_{j_k}\}$ converges locally uniformly in $B(z_0,r_0)$
to $h.$ By the generalized   Hurwitz  Theorem,
either $h\equiv 0$ or $h$ never vanishes. If $h\equiv 0$ it is easy to see that $f_{j_k}(z)\to \infty $
uniformly on compact subsets of $B(z_0,r_0).$ If $h$ never vanishes then $1/h$ is holomorphic function in $B(z_0,r_0)$
and it follows that $f_{j_k}(z)\to  1/h(z)$ uniformly on compact subsets of $B(z_0,r_0).$ Therefore $\mathcal H$ is normal at $z_0.$

Since $\mathcal J$ and $\mathcal H$ are normal at $z_0,$ so that the union $\mathcal F$ is normal  in
$z_0.$
Since normality  is  a  local  property,  $\mathcal F$  is a normal  family in $\Omega.$ This completes the
proof of the theorem.
\end{proof}

\begin{rem} It should be pointed out that the above theorem is not true if the condition ``for each point $z_0 \in \Omega$ there exists a ball $B(z_0,r_0)\subset
\Omega$ such that the  the set of quantities
$L(f,B(z_0,r_0)),$ $f\in \mathcal F,$ is bounded'' is replaced by the condition ``the corresponding family of functions given by $|g(z)|/|g(w)|$ is locally bounded on $\Omega \times \Omega$''
(cf. \cite[Theorem 2.2.8]{Schiff1993}).

To see this, consider the family $\mathcal F:=\{z^j\}_{j=1}^\infty$ of holomorphic functions. If we take  $\mathcal A:=\{z\in \mathbf C : 1/2<|z|<1\}$, then
$\mathcal F|_{\mathcal A}$ is a set of bounded (by 1) zero-free holomorphic functions in $\mathcal A$ so Montel's theorem
guarantees that $\mathcal F$ is normal. It is plain by inspection  that the family $\Big \{\frac{|z|^j}{|w|^j}\Big\}_{j=1}^\infty$ is not locally bounded
on
$\mathcal A\times\mathcal A,$ while $\Big\{\frac{\ln|z|^j}{\ln|w|^j}\Big\}_{j=1}^\infty$ is a locally bounded family on $\mathcal A\times\mathcal A.$ Hence
Theorem 2.2.8 in \cite{Schiff1993} is not true.

\end{rem}

\begin{rem} An alternative proof of necessity in Theorem \ref{Mandelbrojt} can be given: 
Fix a point $z_0$ in $\Omega.$    Suppose that $\mathcal F$ is normal in $\Omega,$ a ball $\overline{B(z_0, r_0)} \subset \Omega,$ but  the set of quantities $L(f,B(z_0,r_0)),$ $f\in \mathcal F,$  is unbounded. Then there exists a sequence $\{f_j\} \subseteq \mathcal F$ such that
\begin{equation}\label{eq}
L(f_j,B(z_0,r_0))>j \textrm{ for all } j\in \mathbb N.
\end{equation}

Choose $r>r_0$ such that $B(z_0,r)\subset \Omega.$
The normality of $\{f_j\}$ provides a subsequence $\{f_{j_k}\} $  converging locally uniformly in $B(z_0,r)$ to a  holomorphic function $f$ or a subsequence $\{f_{j_k}\}$ which converges locally uniformly to  $\infty$ in $B(z_0,r).$
Since each $f_{j_k}$ is  zero-free in $B(z_0,r)$ the function $f$ is either zero-free in $B(z_0,r)$ or $f\equiv 0$  by the generalized Hurwitz theorem.

Therefore for the given sequence $\{f_j\}$ there is a subsequence $\{f_{j_k}\}$ for which one of the following holds:
\begin{itemize}
\item[(a)] $\{f_{j_k}\}$ converges  uniformly on $B(z_0,(r_0+r)/2)$ to the  function $f\equiv 0;$
\item[(b)] $\{f_{j_k}\}$  converges  uniformly on $ B(z_0,(r_0+r)/2)$ to a  holomorphic function $f$ which is zero-free on $B(z_0,r);$
\item[(c)] $\{f_{j_k}\}$  converges  uniformly on $ \overline{B(z_0,(r_0+r)/2)}$ to  $\infty.$
\end{itemize}

Since  $j_k\geq k$ it follows readily from (\ref{eq}) that
\begin{equation}\label{eq1}
L(f_{j_k},B(z_0,r_0))>k \text{ for all } k \in \mathbb N.
\end{equation}

In case (a) (respectively in case (c)) we have $|f_{j_k}(z)|<1/2$ (respectively  $|f_{j_k}(z)|>2$)  for all $z\in\overline{B(z_0,(r_0+r)/2)}$ and all $k \in \mathbb N$ sufficiently large.   It follows that $\ln|f_{j_k}(z)|$  is  a strongly negative (respectively  positive) pluriharmonic function in $B(z_0, (r_0+r)/2).$   Pluriharmonic functions form a subclass of the class
of harmonic functions in  $B(z_0, (r_0+r)/2)$ (obviously proper for $n > 1$). So by
Harnack's inequality  there exists some  constant $C=C(\overline{B(z_0,r_0)},B(z_0, (r_0+r)/2),$ $C\in (1,\infty),$ such that
 $$ \frac{\ln|f_{j_k}(z)|}{\ln|f_{j_k}(w)|}\leq C \textrm{ for all } z \textrm{ and } w \textrm{ in } \overline{B(z_0,r_0)},$$
and hence $m(f_{j_k},B(z_0,r_0)\leq C$ for all $k \in \mathbb N$ sufficiently large.

 In case (b) we have $\lim_{k\to \infty}|f_{j_k}(z)|=f(z)$ for all  $z\in B(z_0,(r_0+r)/2).$ It follows
 $$\lim_{k\to \infty}\frac{|f_{j_k}(z)|}{|f_{j_k}(w)|}=\Big| \frac{f(z)}{f(w)}\Big|\textrm{ uniformly for } z \textrm{ and } w \textrm{ in } \overline{B(z_0,r_0)}.$$
The function $f(z)/f(w)$ is holomorphic  on $B(z_0,r)\times B(z_0,r),$  hence  $m'(f,B(z_0,r_0))$ is bounded. Because $m'(f_{j_k},B(z_0,r_0))\to m'(f,B(z_0,r_0))$ as $k \to \infty,$ we conclude that  there exists some finite constant $C'$ such that
$m'(f_{j_k},B(z_0,r_0))<C'$  for all $k \in \mathbb N$ sufficiently large.

 Since $L(f_{j_k},B(z_0,r_0))$ is the minimum of $m(f_{j_k},B(z_0,r_0))$ and $m'(f_{j_k},B(z_0,r_0))$
 the proof above gives that  the set of quantities $L(f_{j_k},B(z_0,r_0)),$ $k \in \mathbb N,$
is bounded, which is a contradiction to (\ref{eq1}).
\end{rem}

% ----------------------------------------------------------------
\bibliography{myMandelbrojt}

\begin{thebibliography}{1}

\bibitem{Alexander1974}
H.~Alexander.
\newblock Volumes of images of varieties in projective space and in
  grassmannians.
\newblock {\em Transaction of the American Mathematical Society},
  1986:237--249, 1974.

\bibitem{Dovbush1981}
P.V. Dovbush.
\newblock {Normal functions of many complex variables.}
\newblock {\em Mosc. Univ. Math. Bull.}, 36(1):44--48, 1981.

\bibitem{J.Grahl2012}
J\"{u}rgen Grahl and Shahar Nevo.
\newblock Spherical derivatives and normal families.
\newblock {\em Journal d'Analyse Math\'{e}matique}, 117:119--128, 2012.

\bibitem{Julia1926}
G.~Julia.
\newblock Sur les familles de fonctions analytiques de plusieurs variables.
\newblock {\em Acta Mathematica}, 47:53--115, 1926.

\bibitem{Mandelbrojt1929}
S.~Mandelbrojt.
\newblock Sur les suites de fonctions holomorphes. {L}es suites correspondantes
  des fonctions d\'{e}riv\'{e}es. {F}onctions enti\`{e}res.
\newblock {\em Journal de Math\'{e}matiques Pures et Appliqu\'{e}es},
  9(8):173--196, 1929.
\newblock (Also available as \url{
  http://portail.mathdoc.fr/JMPA/afficher_notice.php?id=JMPA_1929_9_8_A10_0 }).

\bibitem{Schiff1993}
J.~L. Schiff.
\newblock {\em {Normal families.}}
\newblock {New York: Springer-Verlag}, 1993.

\bibitem{Timoney1980}
R.~M. Timoney.
\newblock Bloch functions in several complex variables, {I}.
\newblock {\em Bull. London Math. Soc.}, 12(4):241--267, 1980.

\end{thebibliography}
\bibliographystyle{plain}

\end{document}